\documentclass[11pt]{article}

\usepackage{amsfonts}
\usepackage{amssymb,amsmath,amsthm, enumerate}
\usepackage{latexsym}
\usepackage{fullpage}
\usepackage{hyperref}

\newtheorem{theorem}{Theorem}[section]

\newtheorem{lemma}[theorem]{Lemma}
\newtheorem{cor}[theorem]{Corollary}

\theoremstyle{definition}

\newcounter{tenumerate}

\def\P{\mathbb{P}}

\newcommand{\one}{\1}
\newcommand{\deq}{\stackrel{\scriptscriptstyle\triangle}{=}}

\renewcommand{\epsilon}{\varepsilon}

\newcommand{\1}{\mathbf{1}}

\DeclareMathOperator{\var}{Var}

\newcommand{\N}{{\mathbb N}}

\newcommand{\E}{{\mathbb E}}
\newcommand{\remove}[1]{}

\renewcommand{\leq}{\leqslant}
\renewcommand{\geq}{\geqslant}

\newcommand{\cov}{\mathrm{Cov}}

\def\XXint#1#2#3{{\setbox0=\hbox{$#1{#2#3}{\int}$}
\vcenter{\hbox{$#2#3$}}\kern-.5\wd0}}

\begin{document}

\title{{\bf Exponential and double exponential tails for maximum of two-dimensional discrete Gaussian free field}}
\author{Jian Ding\footnote{Most of the work was carried out when the author was supported partially by Microsoft Research.}\\
Stanford University}
\date{}

\maketitle

\begin{abstract}
We study the tail behavior for the maximum of discrete Gaussian free
field on a 2D box with Dirichlet boundary condition after centering
by its expectation. We show that it exhibits an exponential decay
for the right tail and a double exponential decay for the left tail.
In particular, our result implies that the variance of the maximum
is of order 1, improving an $o(\log n)$ bound by Chatterjee (2008)
and confirming a folklore conjecture. An important ingredient for
our proof is a result of Bramson and Zeitouni (2010), who proved the
tightness of the centered maximum together with an evaluation of the
expectation up to an additive constant.
\end{abstract}

\section{Introduction}
Denote by $A_n \subset \mathbb{Z}^2$ a box of side length $n$, i.e., $A = \{(x, y)\in \mathbb{Z}^2: 0\leq x, y\leq n\}$, and let $\partial A_n = \{v\in A_n: \exists u\in \mathbb{Z}^2 \setminus A_n: v\sim u\}$. The discrete Gaussian free field (GFF) $\{\eta_v: v\in A_n\}$ on $A_n$ with Dirichlet boundary condition, is then defined to be a mean zero Gaussian
process which takes value 0 on $\partial A_n$ and satisfies the
following Markov field condition for all $v\in A_n\setminus \partial
A_n$: $\eta_v$ is distributed as a Gaussian variable with variance $1$
and mean equal to the average over the neighbors given the GFF on
$A_n\setminus \{v\}$ (see later for a definition of GFF using Green
functions). Throughout the paper, we use the notation
\begin{equation}\label{eq-def-M-n}
M_n = \sup_{v\in A_n} \eta_v\,.
\end{equation}

We prove the following tail behavior for $M_n$.
\begin{theorem}\label{thm-concentration}
There exist absolute constants $C,c>0$ so that for all $n\in
 \N$  and  $0\leq \lambda\leq (\log n)^{2/3}$
\begin{align*}
c\mathrm{e}^{-C\lambda} &\leq \P(M_n \geq \E M_n + \lambda) \leq
C\mathrm{e}^{-c\lambda}\\
 c\mathrm{e}^{-C \mathrm{e}^{C\lambda}}&\leq \P(M_n \leq \E M_n - \lambda)
\leq C \mathrm{e}^{-c\mathrm{e}^{c\lambda}}\,.
\end{align*}
\end{theorem}

The preceding theorem gives the tail behavior when the deviation is
less than $(\log n)^{2/3}$. For $\lambda \geq (\log n)^{2/3}$, by
isoperimetric inequality for general  Gaussian processes (see, e.g.,
Ledoux \cite[Thm. 7.1, Eq. (7.4)]{Ledoux89}) and the simple fact
that $\max_v\var \eta_v = 2\log n/\pi + O(1)$ (see
Lemma~\ref{lem-green-function}), we have
$$\P(|M_n - \E M_n | \geq \lambda) \leq 2\, \mathrm{e}^{- c\lambda^2/\log n}\,, \mbox{for an absolute constant } c>0\,.$$
Combined with Theorem~\ref{thm-concentration}, this immediately
gives the order of the variance for $M_n$. Before stating the result, let us specify some conventions for notations throughout the paper. The letters $c$ and
$C$ denote absolute positive constants, whose values might vary from
line to line. By convention, we denote by $C$ large constants and by
$c$ small constants. Other absolute constants that appeared are
fixed once and for all. If there exists an absolute constant $C>0$ such that $a_n = C b_n$ for all $n\geq 1$, we write $a_n = O(b_n)$; we write $a_n = \Theta (b_n)$ if $a_n = O(b_n)$ as well as $b_n = O(a_n)$; if $\limsup_{n\to \infty} a_n /b_n \to 0$, we write $a_n = o(b_n)$. We are now ready to state the corollary.
\begin{cor}\label{cor-variance}
We have that $\var M_n = \Theta(1)$.
\end{cor}
Corollary \ref{cor-variance} improves an $o(\log n)$ bound on the
variance due to Chatterjee \cite{Chatterjee08}, thereby confirming a
folklore conjecture (see Question (4) of \cite{Chatterjee08}). An
important ingredient for our proof is the following result on the
tightness of the maximum of the GFF on 2D box due to Bramson and
Zeitouni \cite{BZ10}.
\begin{theorem}\label{thm-BZ}\cite{BZ10}
The sequence of random variables $M_n - \E M_n$ is tight and $$\E
M_n = 2\sqrt{2/\pi} \big(\log n - \tfrac{3}{8\log 2} \log\log n\big)
+ O(1)\,.$$
\end{theorem}

Previously to \cite{BZ10}, Bolthausen, Deuschel and Zeitouni
\cite{BDZ10} proved that $(M_n - \E M_n)$ is tight along a
deterministic subsequence $(n_k)_{k\in \N}$. Earlier works on the
extremal values of GFF include  Bolthausen, Deuschel and Giacomin
\cite{BDG01} who established the asymptotics for $M_n$, and Daviaud
\cite{Daviaud06} who studied the extremes for the GFF.

We compare our results with tail behavior for the maximum of the GFF
on a binary tree. Interestingly, in the case of tree, the maximum
exhibits an exponential decay for the right tail, but a Gaussian
type decay for the left tail as opposed to the double exponential
decay for 2D box. This is because in the case of 2D box, the
Dirichlet boundary condition decouples the GFF near the boundary
such that the GFF behaves almost independently close to the
boundary. The same phenomenon also occurs for the event that all the
GFFs are nonnegative: for a binary tree of height $n$ the
probability is about $\mathrm{e}^{-\Theta(n^2)}$, and for a box of
side length $n$ the probability is about $\mathrm{e}^{-\Theta(n)}$
(see Deuschel \cite{Deuschel96}).

Much more was known about the maximal displacement of branching
Brownian motion (BBM). In their classical paper \cite{KPP37}, Kolmogorov,
Petrovsky, and Piskunov studied its connection with the so-called
KPP-equation, from which it could be deduced that both the right and left tails exhibit exponential types of decay.
The probabilistic interpretation of KPP-equation in
terms of BBM was further exploited by Bramson \cite{Bramson78}. Then the precise asymptotic tails were computed, and in particular a polynomial prefactor for the right tail was detected (this appears to be fundamentally different from the tail of Gumble distribution, which arise from the maximum of, say, i.i.d.\ Gaussian variables). See,
e.g., Bramson \cite{Bramson83} and Harris \cite{Harris99} for the
right tail, and see Arguin, Bovier and Kistler \cite{ABK10} for the
left tail (the argument is due to De Lellis). In addition, Lalley
and Sellke \cite{LS87} obtained an integral representation for the
limiting law of the centered maximum.

\medskip

We now give the definition of GFF using the connection with random
walks (in particular, Green functions). Consider a connected graph
$G = (V, E)$. For $U \subset V$, the Green function $G_U(\cdot,
\cdot)$ of the discrete Laplacian is given by
\begin{equation}\label{eq-def-green-function}
G_U(x, y) = \E_x(\mbox{$\sum_{k=0}^{\tau_U - 1}$} \one\{S_k =
y\})\,, \mbox{ for all } x, y\in V\,,
\end{equation}
where $\tau_U$ is the hitting time to set $U$ for random walk
$(S_k)$, defined by (the notation applies throughout the paper)
\begin{equation}\label{eq-def-tau-A}
\tau_U = \min\{k\geq 0: S_k \in U\}\,.
\end{equation}
The GFF $\{\eta_v: v\in V\}$ with Dirichlet boundary on $U$ is then
defined to be a mean zero Gaussian process indexed by $V$ such that
the covariance matrix is given by  Green function $(G_U(x, y))_{x,
y\in V}$ (In general graph, it is typical to normalize the Green
function by the degree of the target vertex $y$. In the case of 2D
lattices, this normalization is usually dropped since the degrees
are constant). It is clear to see that $\eta_v = 0$ for all $v\in
U$.

\section{Proofs}

In this section, we prove Theorem~\ref{thm-concentration}. We start
with a brief discussion on the proof strategy, and then demonstrate
the upper (lower) bounds for the right (left) tails in the
subsequent four subsections.

\subsection{A word on proof strategy}

Our proof typically employs a two-level structure which involves
either a partitioning or a packing for a 2D box $A_n$ by (slightly)
smaller boxes. In all the proofs, we use Theorem~\ref{thm-BZ} to
control the behavior in small boxes, and study ``typical'' events on
small boxes with probability strictly bounded away from 0 and 1. The
large deviation bounds typically come from gluing the small boxes
together to a big box, with the probability either inverse
proportional to the number of small boxes or exponentially small in
the number of boxes.

By Theorem~\ref{thm-BZ}, there exists a universal constant $\kappa
>0$ such that for all $n \geq 3 n'$
\begin{equation}\label{eq-expectation-difference}
2\sqrt{2/\pi}\log (n/n') - \tfrac{3\sqrt{2/\pi}}{4\log 2}\log(\log
n/\log n') - \kappa \leq \E M_n - \E M_{n'} \leq 2\sqrt{2/\pi}\log
(n/n') + \kappa\,.
\end{equation}
That is to say, in order to observe a difference of $\lambda$ in the
expectation for the maximum, the side length of the box has to
increase (decrease) by a factor of $\exp(\Theta(\lambda))$. This
suggests that the number of small boxes shall be
$\exp(\Theta(\lambda))$ in our two-level structure. Depending on how
the large deviation arises, this will yield a tail of either
exponential or double exponential decay.

In order to construct the two-level structure, we use repeatedly the
decomposition of Gaussian process: for a joint Gaussian process $(X,
Y)$, we can write $X$ as a sum of a (linear) function of $Y$ and an
independent Gaussian process $X'$. Here, we used a crucial fact that
Gaussian processes possess linear structures where orthogonality
implies independence. Furthermore, the next well-known property
specific to GFF proves to be quite useful (see Dynkin \cite[Thm.
1.2.2]{Dynkin80}).
\begin{lemma}\label{lem-DGFF}
Let $\{\eta_v\}_{v\in V}$ be a GFF on a graph $G=(V, E)$. For
$U\subset V$, define $\tau_U$ as in \eqref{eq-def-tau-A}. Then, for
$v\in V$, we have
\begin{align*}\E(\eta_v \mid
\eta_u, u\in U) = \sum_{u\in U}\P_v(S_{\tau_U} = u) \cdot
\eta_u\,.\end{align*}
\end{lemma}

\subsection{Upper bound on the right tail}

In this subsection, we prove that for an absolute constant $C, \lambda_0>0$
\begin{equation}\label{eq-upper-tail}
\P(M_n - \E M_n \geq \lambda) \leq C \mathrm{e}^{-\sqrt{\pi/2}
\lambda}\,, \mbox{ for all } n\in \mathbb{N} \mbox{ and }
\lambda \geq \lambda_0\,.
\end{equation}
Note that we could choose $\lambda_0$ arbitrarily large by adjusting the constant $C$ in Theorem~\ref{thm-concentration}.
Let $N = n \lceil \mathrm{e}^{\sqrt{\pi/8} (\lambda - \kappa -
\alpha)} \rceil$, where $\kappa$ is from
\eqref{eq-expectation-difference} and $\alpha > 0$ will be selected
later. Denote by $p = p_\alpha = \mathrm{e}^{-\sqrt{\pi/2} (\lambda
- \kappa - \alpha)}$ and $k = \lceil \mathrm{e}^{\sqrt{\pi/8}
(\lambda - \kappa - \alpha)} \rceil$. It suffices to prove that
$\P(M_n - \E M_n \geq \lambda) \leq p$, and we prove it by
contradiction. To this end, we assume that
\begin{equation}\label{eq-assumption-upper}
\P(M_n - \E M_n \geq \lambda) > p
\end{equation}
and try to derive a contradiction.

Now, consider an $N \times N$ 2D box $A_N$ and let $\{\eta_v: v\in
A_N\}$ be a GFF on $A_N$ with Dirichlet boundary condition. We
partition $A_N$ into $k^2$ boxes of side length $n$ and denote by
$\mathcal{B}$ the collection of these boxes. We abuse the notation
$\partial \mathcal{B}$ to denote the union of the boundary sets of
the smaller boxes in $\mathcal{B}$. For $B\in \mathcal{B}$, we let
$\{g_v^B: v\in B\}$ be a GFF on $B$ with Dirichlet boundary
condition and we let $\{\{g_v^B: v\in B\}\}_{B\in \mathcal{B}}$ be
independent from each other and independent from $\{\eta_v: v\in
\partial \mathcal{B}\}$. Using the decomposition of Gaussian
process, we can write that for every $v\in B\subseteq A_N$
\begin{equation}\label{eq-decomposition}
\eta_v = g^B_v + \E(\eta_v \mid \{\eta_u: u\in \partial
\mathcal{B}\})\,.
\end{equation} Denote by $\phi_v = \E(\eta_v \mid \{\eta_u: u\in
\partial \mathcal{B}\})$. We note that $\phi_v$ is a convex
combination of $\{\eta_u: u\in
\partial \mathcal{B}\}$ where the linear coefficients are
deterministic. Thus,
\begin{equation}\label{eq-independence}
\{\phi_v : v\in A_N\} \mbox{ is independent of } \{\{g^B_v: v\in
B\}: B\in \mathcal{B}\}\,.
\end{equation}

Denote by $M_B = \sup_{v\in B} g^B_v$. It is clear that $\{M_B : B
\in \mathcal{B}\}$ is a collection of i.i.d.\ random variables and
each of them is distributed as $M_n$. Therefore, by
\eqref{eq-assumption-upper}, we obtain that $\P(M_B \geq \E M_n +
\lambda ) \geq p$.  Using independence, we get
$$\P(\mbox{$\sup_{B\in \mathcal{B}} \sup_{v\in B}$}\, g^B_v \geq \E M_n + \lambda) = \P(\mbox{$\sup_{B\in \mathcal{B}}$} M_B \geq \E M_n + \lambda ) \geq 1/2\,.$$
Let $\chi \in B \subseteq A_N$ such that $g^B_\chi = \sup_{B\in
\mathcal{B}} \sup_{v\in B} g^B_v$. We see that $\chi$ is random
(obviously) and independent of $\{\phi_v : v\in \partial
\mathcal{B}\}$ by \eqref{eq-independence}. Therefore, we obtain
\begin{equation}\label{eq-change-boundary}\P(\mbox{$\sup_{v\in A_N}$} \eta_v \geq \E M_n + \lambda) \geq
\P (g^B_\chi \geq \E M_n + \lambda, \phi_\chi \geq 0) \geq (1/2)
\min_{v\in A_N} \P (\phi_v \geq 0) = 1/4\,.\end{equation} Recalling
\eqref{eq-expectation-difference} and our definition of $N$, we thus
derive that
$$\P(M_N - \E M_N \geq \alpha) \geq 1/4\,.$$
However, Theorem~\ref{thm-BZ} implies that there exists a universal
constant $\alpha(1/4) > 0$ such that $\P(M_n - \E M_n \geq
\alpha(1/4)) < 1/4$ for all $n\in \N$. Setting $\alpha =
\alpha(1/4)$, we arrive at a contradiction and thus show that
\eqref{eq-assumption-upper} cannot hold, thereby establishing
\eqref{eq-upper-tail}.

\subsection{Lower bound on the right tail}
In this subsection, we analyze the lower bound on the right tail and
aim to prove that for absolute constant $c, \lambda_0>0$
\begin{equation}\label{eq-right-lower}
\P(M_n - \E M_n \geq \lambda) \geq  \tfrac{c}{\lambda}\,
\mathrm{e}^{- 8 \sqrt{2\pi}\lambda} , \mbox{ for all } n\in \N \mbox{
and } \lambda_0\leq \lambda \leq (\log n)^{2/3}\,.
\end{equation}
To prove the above lower bound, we consider a box $A_{n'}$ of side
length $n' = n \mathrm{e}^{-\beta \lambda}$ in the center of $A_n$,
where $\beta > 0$ is to be selected (note that since $\lambda \leq (\log n)^{2/3}$, we have $n'\geq 1$ is well defined). Let $\{g_v: v\in A_{n'}\}$ be a
Gaussian free field on $A_{n'}$ with Dirichlet boundary condition
and independent from $\{\eta_v : v\in \partial A_{n'}\}$. Analogous
to \eqref{eq-decomposition}, we can write that
$$\eta_v = g_v + \phi_v, \mbox{ for all } v\in A_{n'}\,,$$
where $\phi_v = \E(\eta_v \mid \{\eta_u : u\in \partial A_{n'}\})$
is a convex combination of $\{\eta_u: u\in \partial A_{n'}\}$. We
wish to estimate the variance of $\phi_v$. For this purpose, we need
the following standard estimates on Green functions for random walks
in 2D lattices. See, e.g., \cite[Prop. 4.6.2, Thm. 4.4.4]{LL10} for
a reference.
\begin{lemma}\label{lem-green-function}
For $A\subset \mathbb{Z}^2$, consider a random walk $(S_t)$ on
$\mathbb{Z}^2$ and define $\tau_{\partial A} = \min\{j\geq 0: S_j
\in
\partial A\}$ be the hitting time to $\partial A$. For $u, v\in A$, let $G_{\partial A}(u, v)$ be the Green function as in \eqref{eq-def-green-function}.
For a certain nonnegative function $a(\cdot, \cdot)$ such that $a(x,
x) = 0$ and $a(x, y) = \frac{2}{\pi} \log |x-y| + \frac{2\gamma \log
8}{\pi} + O(|x-y|^{-2})$, where $\gamma$ is Euler's constant. Then,
we have
$$G_{\partial A}(u, v) = \E_u(a(S_{\tau_{\partial A}}, v)) - a(u,
v)\,.$$
\end{lemma}
By the preceding lemma, we infer that for any $u, w\in \partial
A_{n'}$,
$$\cov(\eta_u, \eta_w) = G_{\partial A_n}(u, w) \geq \tfrac{2}{\pi} \beta\lambda + O(1)\,.$$
Since $\phi_v$ is a convex combination of $\{\eta_u: u\in
\partial A_{n'}\}$, this implies that for all $v\in A_{n'}$
\begin{equation}\label{eq-lower-bound-variance}
\var \phi_v \geq \tfrac{2}{\pi} \beta\lambda + O(1)\,.
\end{equation}
By Theorem~\ref{thm-BZ}, there exists an absolute constant
$\alpha(1/2)$ such that
\begin{equation}\label{eq-def-alpha-1/2}
\P(M_n - \E M_n \geq -\alpha(1/2)) \geq 1/2 \mbox{ for all } n\in \N
\,.\end{equation} Let $\chi\in A_{n'}$ such that $g_\chi =
\sup_{v\in A_{n'}} g_v$. Recalling that $|\E M_n - \E M_{n'}| \leq
2\sqrt{2/\pi} \beta \lambda + O(\log \beta \lambda)+\kappa$ and that $\lambda \geq
\lambda_0$, we obtain that
\begin{align*}\P(\mbox{$\sup_{v\in A_n}$} \eta_v \geq \E M_n + \lambda)& \geq
\P(g_\chi \geq \E M_{n'} - \alpha(1/2), \phi_\chi \geq
\alpha(1/2) + \kappa + (2\sqrt{2/\pi}\beta + 1) \lambda)\\
&\geq \frac{1}{2} \frac{\pi}{\sqrt{\beta \lambda + O(1)}} \int_{z \geq \alpha(1/2) + \kappa + (2\sqrt{2/\pi}\beta + 1) \lambda} \mathrm{e}^{-\frac{z^2}{2\beta \lambda /\pi +O(1)}} dz \\
& \geq \frac{c}{\sqrt{\lambda}} \mathrm{e}^{-\pi (2\sqrt{2/\pi} \beta + 1)^2\lambda/\beta}\,,
\end{align*}
where the first inequality follows from
\eqref{eq-lower-bound-variance} and the independence between $\chi$
and $\{\phi_v: v\in A_{n'}\}$ (analogous to
\eqref{eq-independence}), and in the second inequality $c>0$ is a small absolute constant. Setting $\beta = \sqrt{\pi/8}$, we obtain
the desired estimate \eqref{eq-right-lower}.

\subsection{Upper bound on the left tail}

In this subsection, we give the upper bound for the lower tail of
the maximum and prove the following for absolute constants $C, c, \lambda_0>0$.
\begin{equation}\label{eq-left-upper}
\P(M_n - \E M_n \leq -\lambda) \leq C \mathrm{e}^{-c
\mathrm{e}^{c\lambda}}\,, \mbox{ for all } n\in \N \mbox{ and }
\lambda_0\leq \lambda \leq (\log n)^{2/3}\,.
\end{equation}
Let $\alpha = \alpha(1/2)$ be defined as in
\eqref{eq-def-alpha-1/2}. Denote by $r = n
\exp(-\sqrt{\pi/8}(\lambda - \alpha - \kappa-4))$ and $\ell = n
\exp(-\sqrt{\pi/8}(\lambda - \alpha - \kappa-4)/3)$. Assume that the
left bottom corner of $A_n$ is the origin $o= (0, 0)$. Define $o_i =
(i\ell, 2r)$ for $1\leq i \leq m = \lfloor n/2\ell \rfloor$. Let
$\mathcal{C}_i$ be a discrete ball of radius $r$ centered at $o_i$
and let $B_i \subset \mathcal{C}(i)$ be a box of side length $r/8$
centered at $o_i$. Let $\mathfrak{C} = \{\mathcal{C}_i: 1\leq i\leq
m\}$ and $\mathcal{B} = \{B_i: 1\leq i\leq m\}$. Analogous to
\eqref{eq-decomposition}, we can write
$$\eta_v = g_v^B + \phi_v \mbox{ for all } v\in B \subseteq \mathcal{C} \in \mathfrak{C}\,,$$
where $\{g_v^B: v\in B\}$ is the projection of the GFF on
$\mathcal{C}$ with Dirichlet boundary condition on $\partial
\mathcal{C}$, and $\{\{g_v^B: v\in B\} :  B\in \mathcal{B}\}$ are
independent of each other and of $\{\eta_v : v\in
\partial \mathfrak{C}\}$ (here $\partial \mathfrak{C} = \cup_{\mathcal{C} \in \mathfrak C} \partial \mathcal C$), and $\phi_v = \E(\eta_v \mid \{\eta_u: u\in
\partial \mathfrak{C}\})$ is a convex combination of $\{\eta_u: u\in
\partial \mathfrak{C}\}$. For every $B\in \mathcal{B}$, define
$\chi_B \in B$ such that
$$g_{\chi_B}^B = \sup_{v\in B}g_v^B\,.$$
Recalling \eqref{eq-expectation-difference}, we get that $\E M_n -
\E M_{r/8} \leq \lambda - \alpha$ (here we assume $\lambda_0$ is large enough such that $n > r/8$).

Using an analogous derivation of \eqref{eq-change-boundary}, we get
that
$$\P(g_{\chi_B}^B \geq \E M_n - \lambda) \geq 1/4\,,$$
where we used definition of $\alpha$ in \eqref{eq-def-alpha-1/2}.
Let $W = \{\chi_B: g_{\chi_B}^B \geq \E M_n - \lambda, B\in
\mathcal{B}\}$. By independence, a standard concentration argument
gives that for an absolute constant $c> 0$
\begin{equation}\label{eq-W}
\P(|W| \leq \tfrac{1}{8} m) \leq \mathrm{e}^{-c m}\,.\end{equation}

It remains to study the process $\{\phi_v: v\in W\}$. If there
exists $v\in W$ such that $\phi_v > 0$,  we have $\sup_{u \in A_n}
\eta_u > \E M_n - \lambda$. Thanks to independence, it then suffices
to prove the following lemma.
\begin{lemma}\label{lem-positive-in-U}
Let $U\subset \cup_{B\in \mathcal{B}} B$ such that $|U\cap B| \leq
1$ for all $B\in \mathcal{B}$. Assume that $|U| \geq m/8$. Then, for
some absolute constants $C, c>0$
$$\P(\phi_v \leq 0 \mbox{ for all } v\in U) \leq C \mathrm{e}^{-c \mathrm{e}^{c \lambda}}\,.$$
\end{lemma}

To prove the preceding lemma, we need to study the correlation
structure for the Gaussian process $\{\phi_v: v\in U\}$.
\begin{lemma}\cite[Lemma 6.3.7]{LL10}
For all $n\geq 1$, let $\mathcal{C}(n) \subset \mathbb{Z}^2$ be a discrete ball of
radius $n$ centered at the origin. Then there exist absolute constants
$c, C>0$ such that for all $n\geq 1$ and $x\in \mathcal{C}(n/4)$ and $y\in
\partial \mathcal{C}(n)$
$$c/n\leq \P_x(\tau_{\partial \mathcal{C}(n)} = y) \leq C/n\,.$$
\end{lemma}
Write $a_{v, w} = \P_v(\tau_{\partial \mathcal{C}} = \tau_w)$. The preceding lemma implies that $c/r\leq a_{v, w} \leq C/r$ for all $v\in B\subset \mathcal{C}$. Combined with Lemma~\ref{lem-DGFF}, it follows that
\begin{equation}\label{eq-convex-coeff}
\phi_v = \sum_{w\in
\partial \mathcal{C}} a_{v, w} \eta_w\,.\end{equation} Therefore, we have
\begin{equation}\label{eq-variance}
\var \phi_v = \Theta(1/r^2) \sum_{u, w} \mathrm{Cov}(\eta_u, \eta_w)= \Theta(1/r^2) \sum_{u, w\in \partial \mathcal{C}}
G_{\partial A_n}(u, w)\,.\end{equation}
 In order to estimate the sum of Green
functions, one could use Lemma~\ref{lem-green-function}.
Alternatively, it is computation free if we apply the next lemma.
\begin{lemma}\label{lem-annulus}\cite[Prop. 6.4.1]{LL10}
For all $n\geq 1$, let $\mathcal{C}(n) \subset \mathbb{Z}^2$ be a discrete ball of
radius $n$ centered at the origin. Then for all $k< n$ and $x\in \mathcal{C}(n) \setminus \mathcal{C}(k)$, we have
$$\P_x(\tau_{\partial \mathcal{C}(n)} < \tau_{\partial \mathcal{C}(k)}) = \frac{\log |x| - \log k + O(1/k)}{\log n - \log k}\,.$$
\end{lemma}
Now, write
$$p_{\min} = \min_{\mathcal{C} \in \mathfrak{C}}\min_{u\in \partial \mathcal{C}} \P_u(\tau_{\partial A_n} < \tau^+_{\partial \mathcal{C}}), \mbox{ and }
p_{\max} = \max_{\mathcal{C} \in \mathfrak{C}}\max_{u\in \partial \mathcal{C}} \P_u(\tau_{\partial A_n} < \tau^+_{\partial \mathcal{C}})\,,$$
where $\tau^+_{\partial \mathcal{C}} = \min\{k
\geq 1: S_k \in
\partial \mathcal{C}\}$ is the first returning time to $\partial \mathcal{C}$. By the preceding lemma, we have
$$1/(4 r\lambda) \leq p_{\min} \leq p_{\max} \leq O(1/r)
\mbox{ for all } u\in \partial \mathcal{C} \mbox{ and } \mathcal{C}
\in \mathfrak{C}\,.$$  Therefore, by Markovian property we have
\begin{equation}\label{eq-green-bound}
\Theta(r) \leq \frac{1}{p_{\max}} \leq \sum_{w\in
\partial \mathcal{C}} G_{\partial A_n}(u, w) \leq 1+ \frac{1}{p_{\min}} = O( r\lambda),
\mbox{ for all } u\in \partial \mathcal{C} \mbox{ and } \mathcal{C}
\in \mathfrak{C}\,.\end{equation} Combined with \eqref{eq-variance},this implies that
$$\Theta(1) \leq \var(\phi_v) =  O(\lambda)\,, \mbox{ for all } v\in U\,.$$
We also wish to bound the covariance between $\phi_v$ and $\phi_u$
for $u, v\in U$. Assume $u\in \mathcal{C}_i$ and $v\in
\mathcal{C}_j$ for $i\neq j$. By \eqref{eq-green-bound}, we see that
\begin{align}\label{eq-covariance}
\cov(\phi_u, \phi_v)& \leq O(1/r) \max_{x\in \mathcal{C}_i}
G_{\partial A_n}(x, \partial \mathcal{C}_j)\nonumber \leq O(1/r)
\max_{x\in \mathcal{C}_i} \P_x(\tau_{\partial \mathcal{C}_j} <
\tau_{\partial A_n}) \max_{y\in \partial\mathcal{C}_j} G_{\partial A_n}(y,
\partial \mathcal{C}_j)\nonumber\\
& \leq O(1/r) \max_{x\in \mathcal{C}_i} \P_x(\tau_{\partial \mathcal{C}_j} <
\tau_{\partial A_n}) \max_{y\in \partial\mathcal{C}_j} \sum_{z\in \partial \mathcal{C}_j}G_{\partial A_n}(y,
z)\nonumber\\
& \leq O(\lambda)\max_{x\in \mathcal{C}_i} \P_x(\tau_{\partial \mathcal{C}_j} <
\tau_{\partial A_n})\,.
\end{align}
We incorporate the estimate for the above hitting probability in the
next lemma.
\begin{lemma}\label{lem-hitting-probability}
For any $i\neq j$ and $x\in \mathcal{C}_i$, we have
$$\P_x( \tau_{\partial \mathcal{C}_j} < \tau_{\partial A_n}) \leq C\sqrt{r/\ell}\,,$$
where $C>0$ is a universal constant.
\end{lemma}
\begin{proof}
We consider the projection of the random walk to the horizontal and
vertical axes, and denote them by $(X_t)$ and $(Y_t)$
respectively. Define
$$T_{_X} = \min\left\{t: |X_t - x| \geq \ell/2 \right\}\,, \mbox{ and } T_{_Y} = \min\{t: Y_t = 0\}\,.$$
It is clear that $\tau_{\partial A_n}\leq T_{_Y}$ and $T_{_X} \leq
\tau_{\partial \mathfrak{C} \setminus \partial \mathcal{C}_i}$.
Write $t^\star = r \ell$. Since the number of steps spent on waling in the horizontal (vertical) axis is a Binomial distribution with parameter $t$ and $1/2$, an application of CLT yields that with probability at least $1-
\exp(-c t^\star)$ (here $c>0$ is an absolute constant) the number of such steps is at least $t^\star/3$ (and thus, at
most $2t^\star/3$). Combined with standard estimates for
1-dimensional random walks (see, e.g., \cite[Thm. 2.17, Lemma
2.21]{LPW09}), it follows that for a universal constant $C>0$
$$\P(T_{_Y} \geq t^\star) \leq C \sqrt{r/\ell}\,.$$
Using Markov property for random walk, we see that
$$\P(T_{_X} \leq t^\star) \leq (\P(T_{_X} \leq \ell^2))^{t^\star /\ell^2} \leq \epsilon^{r/\ell}\,,$$
where $\epsilon<1$ is an absolute constant. This completes the proof.
\end{proof}
Combining the preceding lemma and \eqref{eq-covariance}, we obtain
that (here we assume that $\lambda_0$ is large enough)
$$\cov(\phi_u, \phi_v) = O(\lambda \sqrt{r/\ell})\,, \mbox{ for all }u, v\in U\,.$$
Therefore, we have the following bounds on the correlation
coefficients $\rho_{u, v}$:
\begin{equation}\label{eq-correlation-coeff}
0\leq \rho_{u, v} = O(\lambda \sqrt{r/\ell})\,, \mbox{ for all } u \neq v\in
U\,.
\end{equation}
At this point, we wish to apply Slepian's \cite{slepian62}
comparison theorem (see, also, \cite{Fernique75, LT91}).
\begin{theorem}\label{thm-slepian}
If $\{\xi_i: 1\leq i \leq n\}$ and $\{\zeta_i: 1\leq i\leq n\}$ are
two mean zero Gaussian process such that
\begin{equation}\label{eq-assumption}
\var \xi_i = \var \zeta_i, \mbox{ and } \cov (\xi_i, \xi_j) \leq
\cov(\zeta_i, \zeta_j) \mbox{ for all } 1\leq i, j \leq
n\,.\end{equation} Then for all real numbers $\lambda_1, \ldots,
\lambda_n$,
$$\P(\xi_i \leq \lambda_i \mbox{ for all } 1\leq i\leq n) \leq \P(\zeta_i \leq \lambda_i \mbox{ for all } 1\leq i\leq n)\,.$$
\end{theorem}
The following is an immediate consequence.
\begin{cor}
Let $\{\xi_i: 1\leq i\leq n\}$ be a mean zero Gaussian process such
that the correlation coefficients satisfy $0\leq \rho_{i, j} \leq \rho\leq
1/2$ for all $1\leq i< j\leq n$. Then,
$$\P(\xi_i\leq 0, \mbox{ for all } 1\leq i\leq n ) \leq \mathrm{e}^{-1/(2\rho)} + (9/10)^n\,.$$
\end{cor}
\begin{proof}
Since we are comparing $\xi_i$'s with zero, it allows us to assume
that $\var \xi_i = 1$ for all $1\leq i\leq n$. Let $\zeta_i =
\sqrt{\rho} X + \sqrt{1-\rho^2} Y_i$ where $X$ and $Y_i$'s are
i.i.d.\ standard Gaussian variables. It is clear that our processes
$\{\xi_i: 1\leq i\leq n\}$ and $\{\zeta_i: 1\leq i \leq n\}$ satisfy
\eqref{eq-assumption}. By Theorem~\ref{thm-slepian}, we obtain that
$$\P(\xi_i \leq 0 \mbox{ for all } 1\leq i\leq n) \leq \P(\zeta_i \leq 0 \mbox{ for all } 1\leq i\leq n)\,.$$
Since $\{\zeta_i \leq 0 \mbox{ for all } 1\leq i\leq n\} \subseteq
\{X \leq -1/\sqrt{\rho}\} \cup \{Y_i \leq 1/\sqrt{1-\rho^2} \mbox{
for all }1\leq i\leq n\}$, we have
\begin{align*}\P(\zeta_i \leq 0 \mbox{ for all } 1\leq i\leq n)
&\leq \P(X \leq -1/\sqrt{\rho}) + \P(Y_i
\leq 1/\sqrt{1-\rho^2} \mbox{ for all }1\leq i\leq n)\\
& \leq \mathrm{e}^{-1/(2\rho)} + (9/10)^n\,.
\end{align*}
Altogether, this completes the proof.
\end{proof}
\begin{proof}[\emph{\textbf{Proof of Lemma~\ref{lem-positive-in-U}}}]
Recall definitions of $r$, $\ell$ and $m$. The desired estimate
follows from an application of the preceding corollary to $\{\phi_v:
v\in U\}$ and the correlation bounds \eqref{eq-correlation-coeff} (here we assume that $\lambda$ is large enough such that $\rho_{u, v}\leq 1/2$ for all $u\neq v$).
\end{proof}
Combining Lemma~\ref{lem-positive-in-U} and \eqref{eq-W}, we finally
complete the proof for the upper bound on the left tail as in
\eqref{eq-left-upper}.

\subsection{Lower bound on the left tail}
In this subsection, we study the lower bound for the lower tail of
the maximum and show that for absolute constants $C , c , n_0, \lambda_0> 0$
\begin{equation}\label{eq-left-lower}
\P(M_n - \E M_n \leq -\lambda) \geq c \mathrm{e}^{-C
\mathrm{e}^{C\lambda}}\,, \mbox{ for all }n\geq n_0 \mbox{ and } \lambda_0 \leq
\lambda \leq (\log n)^{2/3}\,.
\end{equation}

The proof consists of two steps: (1) We estimate the probability for
$\sup_{v\in B} \eta_v \leq \E M_n - \lambda$ for a small box $B$ in
$A_n$. (2) Applying FKG inequality for GFF, we bootstrap the
estimate on a small box to the whole box.

By Theorem~\ref{thm-BZ}, there exists an absolute constant $\alpha^*
> 0$ such that
\begin{equation}\label{eq-def-alpha*}
\P(M_n \leq \E M_n + \alpha^*) \geq 3/4 \mbox{ for all } n\in
\mathbb{N}\,.
\end{equation}
We first consider the behavior of GFF in a box of side length
$\ell$, where
\begin{equation}\label{eq-def-ell}
\ell \deq n \mathrm{e}^{-10(\lambda + \kappa + \alpha^* +
2)}\,.\end{equation}
\begin{lemma}\label{lem-smaller-box}
Let $B\subseteq A_n$ be a box of side length $\ell$. Then,
$$\P(\sup_{v\in B} \eta_v \leq \E M_n - \lambda) \geq 1/2\,.$$
\end{lemma}
In order to prove the lemma, let $B'$ be a box of side length
$2\ell$ that has the same center as $B$, and let $\hat{B} = B' \cap
A_n$. Consider the GFF $\{g_v: v\in \hat{B}\}$ on $\hat{B}$ with
Dirichlet boundary condition (on $\partial \hat{B}$). We wish to
compare $\{\eta_v: v\in B\}$ with $\{g_v: v\in B\}$. For $u, v\in
B$, let
$$\rho_{u, v} = \frac{\cov(\eta_u, \eta_v)}{\sqrt{\var \eta_u \, \var \eta_v}} \mbox{ and }  \hat{\rho}_{u, v} = \frac{\cov(g_u, g_v)}{\sqrt{\var g_u \, \var g_v}}$$
be the correlations coefficients of two GFFs under consideration.
\begin{lemma}\label{lem-correlation-coeff}
For all $u, v\in B$, we have $\rho_{u, v} \geq \hat{\rho}_{u, v}$
for all $u, v \in B$.
\end{lemma}
\begin{proof}
Since by definition $\hat{B} \subset A_n$, we see that
$\tau_{\partial \hat{B}} \leq \tau_{\partial A_n}$ deterministically
for a random walk started from an arbitrary vertex in $B$. Note that
$$G_{\partial A_n} (u, v) = \P_u(\tau_v < \tau_{\partial A_n}) G_{\partial A_n}(v, v) \,\mbox{ and }\, G_{\partial \hat{B}} (u, v) = \P_u(\tau_v < \tau_{\partial \hat{B}}) G_{\partial \hat{B}}(v, v)$$
Altogether, we obtain that
\begin{equation*}
\rho_{u, v} = \sqrt{\P_u(\tau_v < \tau_{\partial A_n}) \P_v(\tau_u <
\tau_{\partial A_n})} \geq \sqrt{\P_u(\tau_v < \tau_{\partial
\hat{B}}) \P_v(\tau_u < \tau_{\partial \hat{B}})} = \hat{\rho}_{u,
v}\,. \qedhere\end{equation*}
\end{proof}
We next compare the variances for the two GFFs.
\begin{lemma}\label{lem-variances}
For all $v\in B$, we have that
$$\var \eta_v \leq \Big(1 + \frac{(1+o(1)(\log
(n/\ell) + O(1))}{\log n}\Big) \var g_v\,.$$
\end{lemma}
\begin{proof}
It suffices to compare the Green functions $G_{\partial A_n}(v, v)$
and $G_{\partial \hat{B}}(v, v)$. We can decompose them in terms of
the hitting points to $\partial \hat{B}$ and obtain that
$$G_{\partial A_n} (v, v) = G_{\partial \hat{B}}(v, v) + \sum_{w\in \partial \hat{B}}\P_v(\tau_w = \tau_{\partial \hat{B}}) G_{\partial A_n}(w, v)\,.$$
Note that for $w\in \partial \hat{B} \cap \partial A_n$, we have
$G_{\partial A_n}(w, v) = 0$. For $w\in \partial \hat{B} \setminus
\partial A_n$, we see that $|v-w| \geq \ell$ by our
definition of $\hat{B}$. Therefore, by
Lemma~\ref{lem-green-function}, we have
$$G_{\partial A_n}(w, v) \leq \tfrac{2}{\pi} \log (n/\ell) + O(1)\,.$$
Since $|v-w|\geq \ell$ for $w\in \partial \hat{B} \setminus \partial
A_n$, Lemma~\ref{lem-green-function} gives that
$$G_{\partial \hat{B}}(v, v) = \sum_{w\in \partial \hat{B} \setminus \partial A_n}  \P_v(\tau_w = \tau_{\partial \hat{B}}) \cdot a(w, v)  \geq \big(\tfrac{2}{\pi} +o(1)\big) \log n \sum_{w\in \partial \hat{B} \setminus \partial A_n}  \P_v(\tau_w = \tau_{\partial \hat{B}}) \,,$$
where we used the assumption that $\lambda \leq (\log n)^{2/3}$.
Altogether, we get that $$G_{\partial A_n} (v, v) \leq \big(1 +
\tfrac{(1+o(1)) (\log (n/\ell) + O(1))}{\log n}\big) G_{\partial
\hat{B}} (v, v)\,,$$ completing the proof.
\end{proof}
We will need the following lemma to handle some technical issues.
\begin{lemma}\label{lem-monotone}
For a graph $G = (V, E)$, consider $V_1\subset V_2 \subset V$. Let
$\{\eta^{(1)}_v\}_{v\in V}$ and $\{\eta^{(2)}_v\}_{v\in V}$ be GFFs
on $V$ such that $\eta^{(1)}|_{V_1} = 0$ and $\eta^{(2)}|_{V_2} =
0$, respectively. Then for any number $t \in \mathbb{R}$
$$\P(\mbox{$\sup_{v\in U}$}\eta^{(1)}_v \geq t) \geq \tfrac{1}{2} \P(\mbox{$\sup_{v\in U}$}\eta^{(2)}_v \geq t)\,.$$
\end{lemma}
\begin{proof}
Note that the conditional covariance matrix of
$\{\eta^{(1)}_v\}_{v\in U}$ given the values of
$\{\eta^{(1)}_v\}_{v\in V_2\setminus V_1}$ corresponds to the
covariance matrix of $\{\eta^{(2)}_v\}_{v\in U}$. This implies that
$$\{\eta^{(1)}_v: v\in U\}\stackrel{law}{=} \{\eta^{(2)}_v + \E(\eta^{(1)}_v \mid \{\eta^{(1)}_u: u\in V_2 \setminus V_1\}):  v\in U\}\,,$$
where on the right hand side $\{\eta^{(2)}_v: v\in U\}$ is
independent of $\{\eta^{(1)}_u: u\in V_2\setminus V_1\}$. Write
$\phi_v = \E(\eta^{(1)}_v \mid \{\eta^{(1)}_u: u\in V_2 \setminus
V_1\})$. Note that $\phi_v$ is a linear combination of
$\{\eta^{(1)}_u: u\in V_2 \setminus V_1\}$, and thus a mean zero
Gaussian variable. By the above identity in law, we derive that
$$\P(\mbox{$\sup_{v\in U}$}\eta^{(1)}_v \geq t)  \geq \P(\eta^{(2)}_\xi + \phi_\xi \geq t) =  \tfrac{1}{2}\P(\eta^{(2)}_\xi \geq t) = \tfrac{1}{2} \P(\mbox{$\sup_{v\in U}$} \eta^{(2)}_v \geq t)\,,$$
where we denote by $\xi \in U$ the maximizer of $\{\eta^{(2)}_u:
u\in U\}$ and the second transition follows from the independence of
$\{\eta^{(1)}_v\}$ and $\{\phi_v\}$.
\end{proof}
We are now ready to give
\begin{proof}[\emph{\textbf{Proof of Lemma~\ref{lem-smaller-box}}}]
Write $b_v = \sqrt{\var \eta_v / \var g_v}$ for every $v\in B$. By
Lemma~\ref{lem-variances}, we see that $b_v \leq 1 +
(1/2+o(1))(\log(n/\ell)+O(1))/\log n$ for all $v\in B$. Consider the
Gaussian process defined by $\xi_v = \eta_v/b_v$. By
Lemma~\ref{lem-correlation-coeff}, we see that $\{\xi_v: v\in B\}$
and $\{g_v: v\in B\}$ satisfy the assumption in
Theorem~\ref{thm-slepian}, and thus
\begin{equation}\label{eq-prob-compare}
\P(\mbox{$\sup_{v\in B}$} \xi_v \leq \gamma) \geq
\P(\mbox{$\sup_{v\in B}$} g_v \leq \gamma) \,, \mbox{ for all }
\gamma \in \mathbb{R}\,.
\end{equation}
Plugging into $\gamma = \E M_{2\ell} + \alpha^*$ and using
\eqref{eq-def-alpha*} and Lemma~\ref{lem-monotone} (we need to use Lemma~\ref{lem-monotone} as the box $\hat{B}$ might not be a squared box of side-length $2\ell$ but a subset of that), we obtain that
$$\P(\mbox{$\sup_{v\in B}$} \xi_v \leq \E M_{2\ell} + \alpha^*) \geq \P(\mbox{$\sup_{v\in B}$} g_v \leq \E M_{2\ell} + \alpha^*) \geq \P(\mbox{$\sup_{v\in \hat B}$} g_v \leq \E M_{2\ell} + \alpha^*) \geq 1/2\,.$$
Also, By definition of $\ell$ and \eqref{eq-expectation-difference} as well as our assumption that $\lambda \leq (\log n)^{2/3}$,
we see that
$$\E M_n \geq \E M_{2\ell} + 2\sqrt{2/\pi} \log (n/\ell) - 10\,.$$
Therefore, for large constants $\lambda_0, n_0$, we can deduce that
\begin{align*}(1 + (1/2+o(1))(\log(n/\ell)+O(1))/\log n) (\E M_{2\ell} + \alpha^*) \leq \E M_{2\ell} +  \tfrac{2}{3}\log (n/\ell)  +1 \leq \E M_n - \lambda\,,\end{align*}
where we used Theorem~\ref{thm-BZ} and the definition of $\ell$ in \eqref{eq-def-ell}.
 Altogether, we deduce that
\begin{equation*}
\P(\mbox{$\sup_{v\in B}$} \eta_v \leq \E M_n - \lambda) \geq 1/2\,.
\qedhere
\end{equation*}
\end{proof}

Now, we wish to apply FKG inequality and obtain the estimate on the
probability $\sup_{v\in A_n} \eta_v \leq \E M_n - \lambda$. Pitt
\cite{Pitt82} proves that the FKG inequality holds for a Gaussian
process with nonnegative covariances. Since clearly the GFF has
nonnegative covariances, the FKG inequality holds for GFF.

Partition $A_n$ into a union of boxes $\mathcal{B}$ where each of
the boxes is of side length at most $\ell$. We choose $\mathcal{B}$ in
a way such that $|\mathcal{B}|$ is minimized. Clearly,
$|\mathcal{B}| \leq (\lceil n/\ell \rceil)^2$. Observing that the
event $\{\sup_{v\in B} \eta_v \leq \E M_n - \lambda\}$ is decreasing for
all $B\in \mathcal{B}$, we apply FKG inequality and
Lemma~\ref{lem-smaller-box}, and conclude that
$$\P(\mbox{$\sup_{v\in A_n}$} \leq \E M_n - \lambda) \geq \prod_{B\in \mathcal{B}} \P(\mbox{$\sup_{v\in B}$} \leq \E M_n - \lambda) \geq (1/2)^{|\mathcal{B}|}\,.$$
Recalling the definition of $\ell$ as in \eqref{eq-def-ell}, this
completes the proof of \eqref{eq-left-lower}.

\subsection*{Acknowledgements}

We thank Tonci Antunovic for a careful reading of an early
manuscript with valuable suggestions on exposition, thank Ofer
Zeitouni for helpful communications, thank Yuval Peres for locating
reference \cite{Dynkin80}, and thank Sourav Chatterjee for helpful
comments on an earlier version of the manuscript. We also warmly thank the anonymous referees
for numerous useful comments, which lead to a significant improvement in exposition.


\small

\begin{thebibliography}{10}

\bibitem{ABK10}
L.-P. Arguin, A.~Bovier, and N.~Kistler.
\newblock The genealogy of extremal particles of branching brownian motion.
\newblock Preprint, available at \verb|http://arxiv.org/abs/1008.4386|.

\bibitem{BDG01}
E.~Bolthausen, J.-D. Deuschel, and G.~Giacomin.
\newblock Entropic repulsion and the maximum of the two-dimensional harmonic
  crystal.
\newblock {\em Ann. Probab.}, 29(4):1670--1692, 2001.

\bibitem{BDZ10}
E.~Bolthausen, J.-D. Deuschel, and O.~Zeitouni.
\newblock Recursions and tightness for the maximum of the discrete, two
  dimensional gaussian free field.
\newblock {\em Elect. Comm. in Probab.}, 16:114--119, 2011.

\bibitem{Bramson78}
M.~Bramson.
\newblock Maximal displacement of branching {B}rownian motion.
\newblock {\em Comm. Pure Appl. Math.}, 31(5):531--581, 1978.

\bibitem{Bramson83}
M.~Bramson.
\newblock Convergence of solutions of the {K}olmogorov equation to travelling
  waves.
\newblock {\em Mem. Amer. Math. Soc.}, 44(285):iv+190, 1983.

\bibitem{BZ10}
M.~Bramson and O.~Zeitouni.
\newblock Tightness of the recentered maximum of the two-dimensional discrete
  gaussian free field.
\newblock {\em Comm. Pure Appl. Math.}
\newblock to appear.

\bibitem{Chatterjee08}
S.~Chatterjee.
\newblock Chaos, concentration, and multiple valleys.
\newblock Preprint, available at \verb|http://arxiv.org/abs/0810.4221|.

\bibitem{Daviaud06}
O.~Daviaud.
\newblock Extremes of the discrete two-dimensional {G}aussian free field.
\newblock {\em Ann. Probab.}, 34(3):962--986, 2006.

\bibitem{Deuschel96}
J.-D. Deuschel.
\newblock Entropic repulsion of the lattice free field. {II}. {T}he
  {$0$}-boundary case.
\newblock {\em Comm. Math. Phys.}, 181(3):647--665, 1996.

\bibitem{Dynkin80}
E.~B. Dynkin.
\newblock Markov processes and random fields.
\newblock {\em Bull. Amer. Math. Soc. (N.S.)}, 3(3):975--999, 1980.

\bibitem{Fernique75}
X.~Fernique.
\newblock Regularit\'e des trajectoires des fonctions al\'eatoires gaussiennes.
\newblock In {\em \'{E}cole d'\'{E}t\'e de {P}robabilit\'es de {S}aint-{F}lour,
  {IV}-1974}, pages 1--96. Lecture Notes in Math., Vol. 480. Springer, Berlin,
  1975.

\bibitem{Harris99}
S.~C. Harris.
\newblock Travelling-waves for the {FKPP} equation via probabilistic arguments.
\newblock {\em Proc. Roy. Soc. Edinburgh Sect. A}, 129(3):503--517, 1999.

\bibitem{KPP37}
A.~Kolmogorov, I.~Petrovsky, and N.~Piskunov.
\newblock Etude de l'équation de la diffusion avec croissance de la quantité de
  matière et son application à un problème biologique.
\newblock {\em Bulletin Université d'Etat Moscou, Bjul. Moskowskogo Gos.
  Univ.}, 1937.

\bibitem{LS87}
S.~P. Lalley and T.~Sellke.
\newblock A conditional limit theorem for the frontier of a branching
  {B}rownian motion.
\newblock {\em Ann. Probab.}, 15(3):1052--1061, 1987.

\bibitem{LL10}
G.~F. Lawler and V.~Limic.
\newblock {\em Random walk: a modern introduction}, volume 123 of {\em
  Cambridge Studies in Advanced Mathematics}.
\newblock Cambridge University Press, Cambridge, 2010.

\bibitem{Ledoux89}
M.~Ledoux.
\newblock {\em The concentration of measure phenomenon}, volume~89 of {\em
  Mathematical Surveys and Monographs}.
\newblock American Mathematical Society, Providence, RI, 2001.

\bibitem{LT91}
M.~Ledoux and M.~Talagrand.
\newblock {\em Probability in {B}anach spaces}, volume~23 of {\em Ergebnisse
  der Mathematik und ihrer Grenzgebiete (3) [Results in Mathematics and Related
  Areas (3)]}.
\newblock Springer-Verlag, Berlin, 1991.
\newblock Isoperimetry and processes.

\bibitem{LPW09}
D.~A. Levin, Y.~Peres, and E.~L. Wilmer.
\newblock {\em Markov chains and mixing times}.
\newblock American Mathematical Society, Providence, RI, 2009.
\newblock With a chapter by James G. Propp and David B. Wilson.

\bibitem{Pitt82}
L.~D. Pitt.
\newblock Positively correlated normal variables are associated.
\newblock {\em Ann. Probab.}, 10(2):496--499, 1982.

\bibitem{slepian62}
D.~Slepian.
\newblock The one-sided barrier problem for {G}aussian noise.
\newblock {\em Bell System Tech. J.}, 41:463--501, 1962.

\end{thebibliography}
\end{document}